\documentclass[12pt]{amsart}
\usepackage{amsmath, amsthm, amscd, amsfonts}
\usepackage{color}

\newtheorem{theorem}{Theorem}[section]
\newtheorem{lemma}[theorem]{Lemma}

\newtheorem{corollary}[theorem]{Corollary}
\theoremstyle{definition}
\newtheorem{definition}[theorem]{Definition}

\theoremstyle{remark}

\newtheorem{remark}[theorem]{Remark}
\numberwithin{equation}{section}

\newfont{\kh}{msbm10}

\begin{document}
\title[Some remarks on derivations]{Some remarks on derivations
on the algebra of operators in Hilbert pro-C*-bimodules}
\author{Kh. Karimi}
\address{Khadijeh Karimi,
\newline Department of Mathematics, Shahrood University of Technology, P.
O. Box 3619995161-316, Shahrood, Iran.}
\email{kh\underline{\space}karimi5005@yahoo.com}
\author{K. Sharifi}
\address{Kamran Sharifi,
\newline Department of Mathematics, Shahrood University of Technology, P.
O. Box 3619995161-316, Shahrood, Iran.}
\email{sharifi.kamran@gmail.com}

%\thanks{This}


\begin{abstract}
Suppose $A$ is a pro-C*-algebra. Let $L_{A}(E)$ be the pro-C*-algebra of adjointable operators on a
Hilbert $A$-module $E$ and let $K_{A}(E)$ be the closed two sided $*$-ideal of all compact
operators on $E$. We prove that if $E$ be a full Hilbert $A$-module, the innerness of derivations on $K_{A}(E)$
implies the innerness of derivations on $L_{A}(E)$.
We show that if $A$ is a commutative pro-C*-algebra and $E$ is a Hilbert $A$-bimodule then every derivation
on $K_{A}(E)$ is zero.
Moreover, if $A$ is a commutative $\sigma$-C*-algebra and $E$ is a Hilbert $A$-bimodule
then every derivation on $L_{A}(E)$ is zero, too.
\end{abstract}
\maketitle

%%%%%%%%%%%%%%%%%%%%%%%%%%%%%%%%%%%%%%%%%%%%%%%%%%%%%%%%%%%%%%
\section{Introduction}
A {\it derivation} on an algebra $A$ is a linear mapping $D$ from $A$ into itself,
such that $D(ab)=D(a)b+aD(b)$ for all $a,b\in A$. We say that $D$ is {\it inner}
if there exists $x\in A$ such that $D(a)=[a,x]=ax-xa$ for every $a\in A.$
An important question in the theory of derivations is that on which algebras and
under what conditions, a derivation is inner or is identically zero. In 1955,
Singer and Wermer \cite{SM} proved that every derivation on a commutative
Banach algebra maps the algebra into its radical, in particular, every
continuous derivation on a semisimple algebra is identically zero.
It is well-known that the Singer-Wermer theorem is not valid for general topological
algebras \cite [Corollary 2.2.4]{Sakai}.  In 1992, Becker \cite{Bec}
studied derivations on a special class of topological algebras, known
as pro-C*-algebras, and proved that the Singer-Wermer theorem is
true for commutative pro-C*-algebras. In this paper we investigate derivations on
the pro-C*-algebra $L_A(E)$ of all adjointable operators on a Hilbert pro-C*-bimodule $E$.

Let us quickly recall the definition of pro-C*-algebras and Hilbert modules over them.
A pro-C*-algebra is a complete Hausdorff complex
topological $*$-algebra $A$ whose topology is determined by its continuous C*-seminorms in the
sense that the net $\{a_{i}\}_{i \in I}$ converges to $0$ if and only
if the net $\{p(a_i)\}_{i \in I}$ converges to $0$ for every
continuous C*-seminorm $p$ on $A$. For example the algebra $C(X)$ of all continuous
complex valued functions on a compactly generated space or a CW complex $X$, the
cartesian product $\prod _{\alpha\in I}A_{\alpha}$ of C*-algebras $A_{\alpha}$
with the product topology are pro-C*-algebras \cite[\S 7.6]{Fra}. A {\it $\sigma$-C*-algebra}
is a pro-C*-algebra whose topology is determined by a countable family of C*-seminorms.
Pro-C*-algebras appear in the study
of certain aspects of C*-algebras such as tangent algebras of C*-algebras, domain of
closed $*$-derivations on C*-algebras, multipliers of
Pedersen's ideal, and noncommutative analogues of classical Lie groups.
These algebras were first introduced by Inoue \cite{ino} who called them
locally C*-algebras and studied more in \cite{Apostol, Fra, phil1}
with different names. A right {\it pre-Hilbert module} over a pro-C*-algebra
$A$ is a right $A$-module $E$,
compatible with the complex algebra structure, equipped with an
$A$-valued inner product $\langle \cdot , \cdot \rangle_{A}:E
\times E \to A \, \ (x,y) \mapsto \langle x,y
\rangle_{A}$, which is $A$-linear in the second variable $y$
and has the properties:
\[
 \langle x,y \rangle_{A}=\langle y,x \rangle_{A}^{*} , \ {\rm and} \
\langle x,x \rangle_{A} \geq 0 \ \ {\rm with} \
{\rm equality} \ {\rm if} \ {\rm and} \ {\rm only} \
{\rm if} \ x=0.
\]

A right pre-Hilbert $A$-module $E$ is a right Hilbert $A$-module if $E$ is complete with respect to the
topology determined by the family of seminorms $\{
\overline{p}^{A} \}_{p \in S(A)}$  where $\overline{p}^{A}
(\xi) = \sqrt{ p( \langle \xi, \xi \rangle_{A})}$, $ \xi \in E$.
The notion of a left Hilbert $A$-module is defined in an analogous way.
Denote by $\langle E,E\rangle_{A}$ the closure of the linear span of all
$\langle x,y\rangle_{A}, x,y\in E$. We call $E$ is \textit{full} if $\langle E,E\rangle_{A}=A.$
One can always consider any right Hilbert module over pro-C*-algebra $A$ as a
full Hilbert module over pro-C*-algebra $\langle E,E\rangle_{A}$. Hilbert
modules over pro-C*-algebras have been studied in the book \cite{mar1} and the papers \cite{Joitacrossedpro, phil1, SHAMal}.

 In 1992, Becker \cite[Proposition 2]{Bec} proved that every derivation on a
 pro-C*-algebra is continuous. He also proved that commutative pro-C*-algebras
 have no nonzero derivations \cite[Corollary 3]{Bec}. In 1995,
 Phillips \cite[Theorem 3]{phil2} proved that every derivation on a
 pro-C*-algebra is approximately inner. A result of J. Ringrose states that
 every derivation of a C*-algebra $A$ into a Banach $A$-bimodule is continuous.
 Weigt and Zarakas \cite{Zar1} proved some extensions of this result
 in the context of pro-C*-bimodules and complete locally convex bimodules.
 In \cite{Zar}, Zarakas introduced the notion of a Hilbert pro-C*-bimodule
 and proved that every derivation of a pro-C*-algebra into a Hilbert
 pro-C*-bimodule is continuous. For a detailed survey of derivations of pro-C*-algebras, see \cite{FWZ}.

 Let $A$ be a pro-C*-algebra and $E$ be a right Hilbert module over $A$.
 A map $T:E\to E$ is called  adjointable if there exists a map $T^{*}:E\to E$
 with the property $\langle Tx,y\rangle_{A}=\langle x,T^{*}y\rangle_{A}$
 for all $x,y \in E$. The set $L_{A}(E)$ of all adjointable maps on $E$ is a
 pro-C*-algebra with the topology defined by the family of C*-seminorms $\{\tilde{p}\}_{p\in S(A)},$
 in which, $\tilde{p}(T)=\sup\{ \overline{p}^{A}(Tx):  x\in E, \overline{p}^{A}(x)\leq 1\}$.
 The set of all compact operators $K_{A}(E)$ on $E$ is defined as
 the closure of the set of all finite linear combinations of the operators
$\{\theta_{x,y}:\ \theta_{x,y}(\xi)=x\langle y,\xi \rangle_{A} ,\ x,y,\xi\in E\}.$
It is a  pro-C*-subalgebra and a two sided ideal of $L_{A}(E)$.

 Derivations on the algebra of operators on Hilbert C*-modules have been
 recently studied in the papers \cite{ MMJ, SN, SJM}. Li et al. \cite[Theorem 2.3]{Li}
 studied the relation between the innerness of derivations on $K_{A}(E)$ and $L_{A}(E)$
 and proved that if $A$ is a $\sigma$-unital commutative C*-algebra and $E$ is a full
 Hilbert $A$-module, then every derivation on $L_{A}(E)$ is inner if every
 derivation on $K_{A}(E)$ is inner. In this paper, we prove the above result
 in the context of pro-C*-algebras and we show that the assumptions
 of $\sigma$-unitality and commutativity of $A$ are not required in this theorem.
 We also consider derivations on the algebra of operators
 of Hilbert pro-C*-bimodules and we prove that if $A$ is a commutative
 pro-C*-algebra and $E$ is a Hilbert $A$-bimodule, then every derivation
 on $K_{A}(E)$ is zero. Furthermore, we show that if $A$ is a
 commutative $\sigma$-C*-algebra and $E$ is a Hilbert $A$-bimodule,
 then every derivation on $L_{A}(E)$ is zero.

%%%%%%%%%%%%%%%%%%%%%%%%%%%%%%%%%%%

\section{Preliminaries}
Let $A$ be a pro-C*-algebra and let $S(A)$ be the set of all
continuous C*-seminorms on $A$ and $p\in S(A)$. We set
$N_{p}=\{a\in A:\ p(a)=0\}$ then $A_p=A/N_{p}$ is a C*-algebra in the norm induced
by $p$. For $p,q\in S(A)$ with $p\geq q$, the surjective
morphisms $\pi_{pq}: A_p \to A_q$ defined by $\pi_{pq}(a+N_{p})=a+N_{q}$
induce the inverse system $\{A_p; \pi_{pq}\}_{p,q\in S(A), \, p\geq q}$
of C*-algebras and $A = \varprojlim_{p}A_p$, i.e.
the pro-C*-algebra $A$ can be identified with $\varprojlim_{p}A_p$.
The canonical map from $A$ onto $A_p$ is
denoted by $\pi_p$ and $a_{p}$ is reserved to denote $a+N_{p}$.
A morphism of pro-C*-algebras is a continuous morphism of $*$-algebras.
An isomorphism of pro-C*-algebras is a morphism of pro-C*-algebras
which possesses an inverse morphism of pro-C*-algebras.

An approximate unit of a pro-C*-algebra $A$ is an increasing net $\{e_{i}\}_{i\in I}$
of positive elements of $A$ such that $p(e_i)\leq 1$ for all
$i\in I$ and $p\in S(A)$; $p(ae_i-a)\to 0$ and $p(e_ia-a)\to 0$ for all
$p\in S(A)$ and $a\in A$. Any pro-C*-algebra has an approximate unit.

Let $E$ be a left and right Hilbert module over a pro-C*-algebra $A$.
Denote by $_{A}\langle \cdot , \cdot \rangle$ and $\langle \cdot , \cdot \rangle_{A}$
the left and right $A$-valued inner products on $E$. Then there are
two locally convex topologies defined on $E$. One denoted by $\tau^{A}$,
induced by the seminorms $\{\overline{p}^{A}\}_{p\in S(A)}$ corresponding
to its structure as a right Hilbert $A$-module and the other, denoted
by $^{A}\tau$, induced by the seminorms $\{^{A}\overline{p}\}_{p\in S(A)}$,
corresponding to its structure as a left Hilbert $A$-module.
Zarakas \cite[Corollary 3.2]{Zar} proved that if
$_{A}\langle x,y\rangle z=x\langle y,z \rangle_{A}$ for all $x,y,z\in E$
and we assume continuity of the left (resp. right) module action, in the sense that
\begin{align}\label{T}
\overline{p}^{A}(ax)\leq p(a)\overline{p}^{A}(x),\ \ \ ^{A}\overline{p}(xa)\leq {^{A}\overline{p}(x)p(a)},\ \  x\in E, a\in A,
\end{align}
then $\overline{p}^{A}(x)={^{A}\overline{p}(x)}$, for all $x\in E$ and
so two topologies $^{A}\tau$ and $\tau^{A}$ on $E$ coincide. Based on
this fact he defined the notion of Hilbert pro-C*-bimodule as follows.

\begin{definition}\label{def1}
Let $A$ be a pro-C*-algebra and $E$ be a left and right Hilbert $A$-module.
If condition \eqref{T} is satisfied and $_{A}\langle x,y\rangle z=x\langle y,z \rangle_{A}$
for all $x,y,z\in E$, then we say that $E$ is a Hilbert pro-C*-bimodule
over $A$ or a Hilbert $A$-bimodule.
\end{definition}

Suppose $E$ is a right Hilbert $A$-module. For each $p\in S(A), N_{p}^{E}=\{\xi\in E: \ \overline{p}^{A}(\xi)=0\}$
is a closed submodule of $E$ and $E_{p}=E/N_{p}^{E}$ is a right
Hilbert $A_{p}$-module with the action $(\xi+N_{p}^{E})\pi_{p}(a)=\xi a+N_{p}^{E}$
and the inner product $\langle\xi+N_{p}^{E},\eta+N_{p}^{E}\rangle_{A_{p}}=\pi_{p}(\langle\xi,\eta\rangle_{A}).$
The canonical map from $E$ onto $E_{p}$ is denoted by $\sigma_{p}^{E}$ and
$\xi_{p}$ is reserved to denote $\sigma_{p}^{E}(\xi).$
For $p,q\in S(A)$ with $p\geq q$, the surjective morphisms
$\sigma_{pq}^{E}:E_{p} \to E_q$ defined by
$\sigma_{pq}^{E}(\sigma_{p}^{E}(\xi))=\sigma_{q}^{E}(\xi)$
induce the inverse system $\{E_{p};\  A_p;\  \sigma_{pq}^{E},\ \pi_{pq} \}_{p,q\in S(A), \, p\geq q}$\,
of right Hilbert C*-modules.
In this case, $\varprojlim_{p}E_{p}$ is a right
Hilbert $A$-module which can be identified with $E$.
Let $E$ and $F$ be right Hilbert $A$-modules and $T:E\to F$ be an $A$-module map.
The module map $T$ is called bounded if for each $p\in S(A)$,
there is $k_p>0$ such that $\overline{p}^{A}_{F}(Tx)\leq k_{p}\ \overline{p}^{A}_{E}(x)$ for all $x\in E$.
The set $L_{A}(E,F)$ of all bounded adjointable $A$-module maps from $E$
into $F$ becomes a locally convex space with the topology defined by the
family of seminorms $\{\tilde{p}\}_{p\in S(A)}$, where $\tilde{p}(T)=\|(\pi_{p})_{*}(T)\|_{ L_{A_p}(E_p,F_p)}$
and $(\pi_{p})_{*}:L_{A}(E,F)\to L_{A_p}(E_p,F_p)$ is defined by
$(\pi_p)_{*}(T)(\xi+N_{p}^{E})=T\xi+N_{p}^{F}$, for
all $T\in L_{A}(E,F)$ and $ \xi\in E$. For $p,q\in S(A)$ with $p\geq q$,
the morphisms $(\pi_{pq})_{*}:L_{A_p}(E_p,F_p)\to L_{A_q}(E_q,F_q)$
defined by $(\pi_{pq})_{*}(T_{p})(\sigma_{q}^{E}(\xi))=\sigma_{pq}^{F}(T_p(\sigma_{p}^{E}(\xi)))$ induce
the inverse system $\{ L_{A_p}(E_p,F_p);\ (\pi_{pq})_{*}\}_{p,q\in S(A), \, p\geq q}$
of Banach spaces such that $\varprojlim_{p} L_{A_p}(E_p,F_p)$ can be identified to $L_{A}(E,F)$.
In particular, topologizing, $L_{A}(E,E)$ becomes a
pro-C*-algebra which is abbreviated by $L_{A}(E)$.
The set of all compact operators $K_{A}(E)$ on $E$ is defined as the closed
linear subspace of $L_A(E)$ spanned by
$\{\theta_{x,y}: \ \theta_{x,y}(\xi)=x\langle y,\xi\rangle ~ {\rm for~ all}~ x,y,\xi \in E\}$.
 This is a pro-C*-subalgebra and a two sided ideal of $L_{A}(E)$, moreover,
$K_{A}(E)$ can be identified to $\varprojlim_{p} K_{A_p}(E_p)$.

\section{Main results}
The following theorem states that the innerness of derivations on $K_{A}(E)$
implies the innerness of derivations
on $L_{A}(E)$. By this theorem, the assumption of $\sigma$-unitality and
commutativity of $A$ in \cite[Theorem 2.3]{Li} can be removed.
%----------------------------------------------------------------------------------------%
\begin{theorem}
Let $A$ be a pro-C*-algebra and let $E$ be a full right Hilbert $A$-module.
If every derivation on $K_{A}(E)$ is inner, then any derivation on $L_{A}(E)$ is also inner.
\end{theorem}
%----------------------------------------------------------------------------------------%
\begin{proof}
Let $D$ be a derivation on $L_{A}(E)$. We claim that $D$ maps $K_{A}(E)$ to itself.
To see this, suppose $x,y\in E$ and apply \cite[Corollary 1.3.11]{mar1}, there exist $a\in A$
and $z\in E$ such that $x=za$. Since $E$ is full, there exists a sequence $\{a_{n}\}$
in $\langle E,E \rangle$ such that $p(a_{n}-a) \to 0$ for all $p\in S(A)$.
Let $a_{n}=\sum_{i=1}^{k_{n}}\langle x_{in},y_{in}\rangle$ where $x_{in},y_{in}\in E$. We have

\begin{eqnarray*}
D(\theta_{za_{n},y})=\sum_{i=1}^{k_{n}}D(\theta_{z\langle x_{in},y_{in} \rangle,y}) &=&
\sum_{i=1}^{k_{n}}D(\theta_{z, x_{in}}\theta_{y_{in},y})
\\ &=& \sum_{i=1}^{k_{n}}D(\theta_{z, x_{in}})\theta_{y_{in},y}+
\sum_{i=1}^{k_{n}}\theta_{z, x_{in}}D(\theta_{y_{in},y})\in K_{A}(E).
\end{eqnarray*}
Let $p\in S(A)$ and $w\in E$ such that $\overline{p}^{A}(w)\leq 1$ then
\begin{eqnarray*}
\overline{p}^{A}((\theta_{za_{n},y} - \theta_{za,y})(w))^{2} &=&
p(\langle y,w \rangle^{*})p(\langle za_{n}-za , za_{n}-za \rangle) p(\langle y,w \rangle)
\\ &\leq & \overline{p}^{A}(y)^{2}\overline{p}^{A}(z)^{2}p(a_{n}-a)^{2}
\end{eqnarray*}
which implies that $\tilde{p}(\theta_{za_{n},y}-\theta_{x,y})\to 0$.
Consequently, $\theta_{za_{n},y}\to \theta_{x,y}$ in $L_{A}(E)$ with
respect to the topology generated by $\{\tilde{p}\}_{p\in S(A)}$.
It follows from \cite[Proposition 2]{Bec} that $D$ maps $K_{A}(E)$ into itself.
Since every derivation on $K_A(E)$ is inner, there exists $T\in K_{A}(E)$
such that $D(K)=TK-KT$, for all $K\in K_{A}(E)$. Let $S\in L_{A}(E)$
and $x,y\in E$ then $D(S\theta_{x,y})=TS\theta_{x,y}-S\theta_{x,y}T$. The latter equality and the fact that
$$D(S\theta_{x,y})=D(S)\theta_{x,y}+SD(\theta_{x,y})=D(S)\theta_{x,y}+ST\theta_{x,y}-S\theta_{x,y}T,$$
imply $D(S)\theta_{x,y}=TS\theta_{x,y}-ST\theta_{x,y}$. Let $w\in E$
then there exist $z\in E$ and  $a\in A$ such that $w=za$. Since $E$ is full,
there exists a sequence $\{a_{n}\}$ in $\langle E,E \rangle$ such
that $p(a_{n}-a)\to 0$ for all $p\in S(A)$.
Let $a_{n}=\sum_{i=1}^{k_{n}}\langle x_{in},y_{in}\rangle$ then
\begin{eqnarray*}
D(S)(za_{n})=\sum_{i=1}^{k_{n}}D(S)(z\langle x_{in},y_{in}\rangle) &=& \sum_{i=1}^{k_{n}}D(S)\theta_{z,x_{in}}(y_{in})
\\ &=& \sum_{i=1}^{k_{n}}(TS\theta_{z,x_{in}}-ST\theta_{z,x_{in}})(y_{in})
\\ &=& (TS-ST)(za_{n}).
\end{eqnarray*}
For all $p\in S(A)$, $\overline{p}_{A}(za_{n}-za)\leq \overline{p}_{A}(z)p(a_{n}-a)$
and so $za_{n}\to za$ in $E$. Hence, continuity of $D(S)$ and $TS-ST$ implies
that $D(S)(w)=(TS-ST)(w)$ for all $w\in E$, that is, $D(S)=TS-ST$.
\end{proof}

%----------------------------------------------------------------------------------------%
\begin{remark}\label{rem}
Let $A$ be a $\sigma$-C*-algebra and let $E$ be a Hilbert $A$-bimodule. Denote by $_{A}I$
the closure of the two-sided ideal
${\rm span} \{_{A}\langle x,y \rangle: x,y\in E \}$. By \cite[Lemma 6.4]{Zar}, $_{A}I$
has an approximate unit $\{u_{\alpha}\}$ with
$u_{\alpha}=\sum_{i=1}^{n}{_{A}\langle x_{i}^{\alpha},x_{i}^{\alpha}\rangle}$,
where $\alpha=\{y_{1}, ... , y_{n}\}\subset E$ ranges over finite subsets
of $E$ and $x_{i}^{\alpha}=(\sum_{j=1}^{n} {_{A}\langle} y_{j},y_{j}\rangle
+\frac{1}{n} \it{1})^{-\frac{1}{2}} \, y_{i}$, $i=1, ... , n$, in which $1$ is
the identity of the unitization of $A$.
Zarakas considered a realization of $K_{A}(E)$, through
the closed two-sided $*$-ideal $_{A}I$.
Indeed, he proved that $ {_{A}I}\simeq K_{A}(E)$ is a topological $*$-isomorphism \cite[Theorem 6.5]{Zar}.
This theorem can be regarded as an extension of a result of Brown, Mingo
and Shen \cite[proposition 1.10]{BMS} for Hilbert C*-bimodules. In view of this
fact, it is easy to see that $\{T_{\alpha}\}$, where
$T_{\alpha}=\sum_{i=1}^{n}\theta_{x_{i}^{\alpha},x_{i}^{\alpha}}$, is an approximate unit for $K_{A}(E)$.
\end{remark}
%___________________________________________________________%

\begin{lemma} \label{lem}
Every derivation of a pro-C*-algebra annihilates its center.
\end{lemma}
%----------------------------------------------------------------------------------------%
\begin{proof}
Let $D$ be a derivation on a pro-C*-algebra $A$. For each $p\in S(A)$, we
consider the linear map  $D_{p}:A_{p}\to A_{p}$ defined by $D_{p}(\pi_{p}(a))=\pi_{p}(D(a))$,
where $a\in A$. Let $a\in N_{p}$. By \cite [Lemma 1]{Bec}, there exist
elements $a_{k}\in N_{p}, k=1,...,4$ such that $a=\sum_{k=1}^{4}i^{k}a_{k}^{2}$.
Since $p(D(a))\leq \sum_{k=1}^{4}p(D(a_{k})a_{k}+a_{k}D(a_{k}))=0$,
we have $D(a)\in N_{p}$ which shows that $D_{p}$ is well-defined.
It is easy to see that $D_{p}$ is a derivation on $A_{p}$.
Denote by $Z(A)$ the center of $A$ and let $a\in Z(A)$.
Then $\pi_{p}(a)$ is
in the center of $A_{p}$, for every $p\in S(A)$. It is well-known
that every derivation on a C*-algebra annihilates its center
\cite[Theorem 2]{Kad}, so $D_{p}(\pi_{p}(a))=0$
for every $p\in S(A)$. It follows that $\pi_{p}(D(a))=0$
for all $p\in S(A)$. Since $S(A)$ is a separating family of C*-seminorms,
we have $D(a)=0$.
\end{proof}

%---------------------------------------------------------------------------------------%
\begin{remark}\label{prop}
Let $A$ be a pro-C*-algebra and let $E$ be a Hilbert $A$-bimodule. If $A$ is commutative, then $K_{A}(E)$ is commutative.
Indeed, for every $x,y,z,u,v\in E$ we have
\begin{eqnarray*}
\theta_{x,y}\theta_{u,v}(z)=\theta_{x.\langle y,u \rangle_{A},v}(z)=
x.\langle y,u \rangle_{A}\langle v,z \rangle_{A}&=&{_{A}\langle x,y \rangle}u.\langle v,z \rangle_{A}
\\ &=& {_{A}\langle x,y \rangle}{_{A}\langle u,v \rangle}.z
\end{eqnarray*}
and
\begin{eqnarray*}
\theta_{u,v}\theta_{x,y}(z)=\theta_{u.\langle v,x \rangle_{A},y}(z)=
u.\langle v,x \rangle_{A}\langle y,z \rangle_{A}&=&{_{A}\langle u,v \rangle}x.\langle y,z \rangle_{A}
\\ &=&{_{A}\langle u,v \rangle}{_{A}\langle x,y \rangle}.z.
\end{eqnarray*}
Therefore $\theta_{x,y}\theta_{u,v}=\theta_{u,v}\theta_{x,y}$ for all $x,y\in E$, which implies
the commutativity of $K_{A}(E)$.
%Since $K_{A}(E)$ is a closed linear span of $\{ \theta_{x,y}: x,y\in E \}$, we can follow that $K_{A}(E)$ is commutative.
\end{remark}
%---------------------------------------------------------------------------------------%

\begin{corollary}\label{cor}
Let $A$ be a pro-C*-algebra and let $E$ be a Hilbert $A$-bimodule. If $A$ is commutative then every derivation on $K_{A}(E)$ is zero.
\end{corollary}
%----------------------------------------------------------------------------------------%
The proof follows from Proposition \ref{prop}, Lemma \ref{lem} and the fact that $K_{A}(E)$ is a pro-C*-algebra.

%---------------------------------------------------------------------------------------%
\begin{theorem}
Let $A$ be a commutative $\sigma$-C*-algebra and let $E$ be a Hilbert $A$-bimodule. Then every derivation on $L_{A}(E)$ is zero.
\end{theorem}
%----------------------------------------------------------------------------------------%
\begin{proof}
Let $D$ be a derivation on $L_{A}(E)$ and let $\{T_{\alpha}\}$ be as in Remark \ref{rem}.
First we show that $D$ maps $K_{A}(E)$ to itself. Let $x,y\in E$ then $D(\theta_{x,y}T_{\alpha})=\sum _{i=1}^{n}D(\theta_{x,y}\theta_{x_{i}^{\alpha},x_{i}^{\alpha}})=
\sum_{i=1}^{n}D(\theta_{x,y})\theta_{x_{i}^{\alpha},x_{i}^{\alpha}}+
\sum_{i=1}^{n}\theta_{x,y}D(\theta_{x_{i}^{\alpha},x_{i}^{\alpha}})
\in K_{A}(E)$. Continuity of $D$ and closeness of $K_{A}(E)$ in $L_{A}(E)$
imply that $D(\theta_{x,y})\in K_{A}(E)$ for all $x,y\in E$. Consequently,
the restriction of $D$ to $K_{A}(E)$ is a derivation on $K_{A}(E)$ which
is zero by Corollary \ref{cor}. Let $S\in L_{A}(E)$ then
$D(S)\theta_{x,y}=D(S\theta_{x,y})-SD(\theta_{x,y})=0$ for all $x,y\in E$.
In particular, $D(S)\theta_{x,D(S)x}=0$ and so
\begin{eqnarray*}
D(S)(x)\langle D(S)(x),D(S)(x) \rangle &=& D(S)(x\langle D(S)(x),D(S)(x) \rangle)
\\ &=& D(S)\theta_{x,D(S)(x)}(D(S)(x))
\\ &=& 0.
\end{eqnarray*}

Therefore, $|D(S)(x)|^{4}=\langle D(S)(x)\langle D(S)(x),D(S)(x) \rangle,
D(S)(x)\rangle=0$, i.e. $D(S)(x)=0$, for all $x\in E$.
\end{proof}

{\bf Acknowledgement}:
The authors would like to thank
the referee for his/her careful reading and useful comments.
%----------------------------------------------------------------------------------------%

\end{document}